\documentclass{amsart}
\usepackage{latexsym,amsxtra,amscd,ifthen}
\usepackage{amsfonts}
\usepackage{pifont}
\usepackage{verbatim}
\usepackage{amsmath}
\usepackage{graphicx}
\usepackage{amsthm}
\usepackage{amssymb}
\usepackage{url,color}
\usepackage[arrow,matrix]{xy}
\usepackage{pict2e,epic,eepic,pstricks}

\usepackage{anysize}\marginsize{25mm}{25mm}{30mm}{38mm}
\allowdisplaybreaks[4]
\newtheorem{thm}{Theorem}[section]

\newtheorem{cor}[thm]{Corollary}
\newtheorem{lem}[thm]{Lemma}

\newtheorem{notation}[thm]{Notation}

\newtheorem{prop-def}[thm]{Proposition-Definition}
\theoremstyle{definition}

\theoremstyle{remark}
\newtheorem{rmk}[thm]{\bf Remark}

\numberwithin{equation}{section}
\def\N{\mathbb{N}}

\def\S{\mathbb{S}}

\def\Ker{{\rm Ker}}

\usepackage{bbm}
\usepackage{float}

\newcommand{\lr}[1]{\langle #1 \rangle}
\newcommand{\ucr}[1]{\underset{#1}{\circ}}

\newcommand{\up}[1]{{^{#1}{\it \!\Upsilon}}}

\def \ot{\otimes}

\def \As{\!u\mathcal{A}ss}
\def \nias{\!\mathcal{A}ss}

\def \I{\mathcal{I}}

\def \ip{{\mathcal{P}}}

\def \ias{\!u\mathcal{A}ss}

\def\gkdim{\operatorname{GKdim}}

\def \Ker{\operatorname{Ker}}
\def \Id{\operatorname{Id}}

\def \1{\mathbbm 1}

\def \k{\Bbbk}

\def \N{\mathbb{N}}
\def \S{\mathbb{S}}

\def \deg{\operatorname{deg}}

\begin{document}
\title{Ideals of the associative algebra operad}

\author{Y.-H. Bao}
\address{(Bao) School of Mathematical Sciences, Anhui University, Hefei 230601, China}
\email{baoyh@ahu.edu.cn}

\author{J.-N. Xu}
\address{(Xu) School of Mathematical Sciences, Anhui University, Hefei 230601, China}
\email{jnxu1@outlook.com}

\author{Y. Ye}
\address{(Ye) School of Mathematical Sciences,
	University of Sciences and Technology of China, Hefei 230026, China}
\email{yeyu@ustc.edu.cn}

\author{J.J. Zhang}
\address{(Zhang) Department of Mathematics, Box 354350,
University of Washington, Seattle, Washington 98195, USA}
\email{zhang@math.washington.edu}

\author{Y.-F. Zhang}
\address{(Zhang) School of Mathematical Sciences, Anhui University,
Hefei 230601, China}
\email{yfzhang9804@163.com}




\begin{abstract}
We prove a one-to-one correspondence between the operadic 
ideals of the operad $\As$ and $T$-ideals. As a consequence,
we show that $\As$ is noetherian and that every proper 
operadic ideal of $\ias$ is generated by a single element.
\end{abstract}

\maketitle

\section{Introduction}
\label{xxsec1}

Throughout let $\Bbbk$ be a base field of characteristic zero.
Most algebraic objects are over $\Bbbk$. Unless otherwise stated 
we consider associative algebras with unit in this paper. A 
\textit{polynomial} \textit{identity} of an algebra $A$ is a 
noncommutative polynomial $f(x_1, \cdots, x_n)$ such that 
$f(a_1, \cdots, a_n)=0$ for all $a_1, \cdots, a_n\in A$. An 
algebra satisfying a nontrivial polynomial identity is called a 
\textit{PI-algebra}. Commutative algebras, the matrix algebra 
over a commutative algebra, finite-dimensional algebras, and 
Grassmann (or exterior) algebras are examples of PI-algebras. 

Polynomial identities of a given PI-algebra were firstly 
investigated by Amitsur and Levisky in \cite{AL} where they 
proved that the standard polynomial of degree $2m$ is an 
identity of minimal degree for the $m\times m$ full matrix 
algebra.  It is well-known that the set of all identities 
satisfied by a PI-algebra is a T-ideal of the free algebra 
$\Bbbk\langle X\rangle$ in countable indeterminants 
$X:=\{x_i\}_{n\in \N_+}$. 

Let $\ias$ (resp. $\nias$) denote the symmetric operad encoding 
the unital associative algebras (resp. the associative algebras 
without unit). It is well-known that the PI-theory such as the 
study of multilinear polynomial identities is related to ideals 
of the operad $\ias$. One motivation of this paper is to spell 
out explicitly some connections between the operad $\ias$ and 
the PI-theory.

Observe that the subspace of a T-ideal consisting of all
multilinear polynomials is essentially equivalent to an
operadic ideal of $\ias$ [Lemma \ref{xxlem2.3}]. Recall
that $\As(n)= \Bbbk \S_n$ for all $n\geq 0$. Let $V_n$ be the
space consisting of all multilinear polynomials in $n$ variables
$x_1, \cdots, x_n$. Clearly, $V_n$ admits an action of the
symmetric group $\S_n$, which is naturally isomorphic to the
regular representation of $\S_n$. Let $A$ be a PI-algebra and 
$V_n(A)$ the subspace of $V_n$ of those polynomials that are 
identities of $A$. Then we have $V_n(A)\cong \I_A(n)$ 
where $\I_A$ is the operadic ideal of $\As$ determined by 
the algebra $A$ [Lemma \ref{xxlem2.3}]. As a consequence,
$V_n/V_n(A)\cong (\ias/\I_A)(n)$. In this situation, we also 
say that $A$ is a PI-algebra associated to the operadic ideal 
$\I_A$. Given a T-ideal $J$ and let $A:=\Bbbk\lr{X}/J$. Then, 
following the above procedure we can construct the associated 
operadic ideal $\I_A$ of $\ias$, denoted by $\Psi(J)$. 
Conversely, for every nonzero operadic ideal $\I$ of 
$\ias$, a $\ias/\I$-algebra is a PI-algebra since each nonzero 
element in $\I$ gives an identity of $A$.  In particular, for 
any vector space $V$, the free $\ias/\I$-algebra $(\ias/\I)(V)$ 
is a PI-algebra associated to $\I$. 

\begin{thm}[Theorem \ref{xxthm2.9}]
\label{xxthm1.1}
There is a natural one-to-one correspondence
$$
\{{\text{proper T-ideals of $\Bbbk\lr{X}$}}\}
\longleftrightarrow \{{\text{proper operadic ideals of $\As$}}\}$$
via the map $J \mapsto \Psi(J)$.
\end{thm}

Theorem \ref{xxthm1.1} says that there is no essential difference 
between T-ideals of $\Bbbk\lr{X}$ and operadic ideals of $\ias$. In 
\cite{BYZ}, the authors studied the ideal structure of 2-unitary
operads similar to $\ias$. Note that $\ias$ is denoted as $\nias$
in \cite{BYZ}. Recall that an operad $\ip$ is \textit{artinian} 
(resp. \textit{noetherian}) if the set of ideals of $\ip$ 
satisfies the descending (resp. ascending) chain condition. Let 
$\ip$ be a locally finite 2-unitary operad. Then
\begin{center}
$\gkdim\ip<\infty \iff \ip$ is artinian $\Longrightarrow \ip$
is noetherian.
\end{center}

It is easily seen that $\ias$ is not artinian since
$\gkdim(\ias)=\infty$. Applying Kemer's theorem 
[Theorem \ref{xxthm2.2}] and the 
relationship between T-ideals and operadic ideals of $\ias$ 
[Theorem \ref{xxthm1.1}], we obtain the following.

\begin{thm}
\label{xxthm1.2}
\begin{enumerate}
\item[(1)]
The operad \; $\ias$ is noetherian.
\item[(2)]
Every proper operadic ideal of $\ias$ is generated by a single
element.
\end{enumerate}
\end{thm}

We are wondering if there is a version of Theorem \ref{xxthm1.2}
for other operads such as unital Poisson operad. Recall that 
$\nias$ is the symmetric operad encoding the associative algebras 
without unit. Note that Theorem \ref{xxthm1.2}(2) fails for $\nias$ 
though Theorem \ref{xxthm1.2}(1) holds [Remark \ref{xxrem2.15}]. An 
operadic ideal $\I$ of $\ias$ may not be generated by an element 
in $\I$ of the minimal degree. For example, the $k$-th truncation 
ideal $\up{k}$ of $\ias$ \cite[E0.0.2]{BYZ} is generated by an 
element in $\up{k}(m)$ for $m>k\ge 3$, rather than in $\up{k}(k)$, 
see \cite{BFXYZZ}. Therefore, it is reasonable to consider the 
single generator of an operadic ideal of $\I$ and the 
corresponding multilinear polynomial in PI-theory, which will 
be studied in \cite{BFXYZZ}.

\noindent\textbf{Acknowledgments}.
We would like to thank Xiao-Wu Chen, Ji-Wei He, and Zerui Zhang 
for many conversations and their useful suggestions. Y.-H. Bao 
was partially supported by the National Natural Science Foundation 
of China (Nos. 11871071 and 12371015) and the Science Fundation for Distinguished 
Young Scholars of Anhui Province (No. 2108085J01). Y. Ye was 
partially supported by the National Natural Science Foundation of 
China (Nos. 12131015, 11971449, 12161141001, and 12371042) and the Innovation Program for Quantum Science and Technology (No. 2021ZD0302902). J.J. Zhang was 
partially supported by the US National Science Foundation 
(Nos. DMS-2001015 and DMS-2302087).

\section{Proofs of statements}
\label{xxsec2}

Throughout $\k$ is a fixed field of characteristic
zero and all unadorned $\ot$ will be $\ot_\k$. First we recall 
some basics about the operad $\ias$. Generally we refer to 
\cite{LV,BYZ, QXZZ} for basic definitions and properties about 
operads. For convenience, we denote $[n]=\{1, 2, \cdots, n\}$.

Let $\S_n$ be the symmetric group of degree $n$. We follow
that convention in \cite{BYZ} and use the sequence
$(\sigma^{-1}(1), \sigma^{-1}(2), \cdots, \sigma^{-1}(n))$
to denote an element $\sigma\in \S_n$. Equivalently, each
$(i_1, i_2, \cdots, i_n)$ of $[n]$ corresponds to the
permutation $\sigma\in \S_n$ given by $\sigma(i_k)=k$
for all $1\le k\le n$. We also use $1_n$ to denote the
identity element in $\S_n$.

Recall that $\ias(n)= \Bbbk\S_n$ is the right regular
$\Bbbk\S_n$-module, and the composition map of $\ias$ is
linearly extended by the following maps: for $n>0$,
$k_1, k_2, \cdots, k_n\ge 0$,
\begin{align*}
\S_{n} \times \S_{k_1}\times \cdots \times \S_{k_n}
   &\to \S_{\sum_{i=1}^n k_i},\\
(\sigma, \sigma_1, \cdots, \sigma_n)
  &\mapsto (\tilde{B}_{\sigma^{-1}(1)}, \cdots,
	\tilde{B}_{\sigma^{-1}(n)})
\end{align*}
for all $\sigma\in \S_n$ and $\sigma_i\in \S_{k_i}$,
$1\le i\le n$, where
\[\tilde{B}_{i}=(\sum_{j=1}^{i-1}k_j+\sigma_i^{-1}(1),
\cdots, \sum_{j=1}^{i-1}k_j+\sigma_i^{-1}(k_i))\]
for all $i=1, \cdots, n$. The partial composition
\[\ias(m) \ucr{i} \ias(n)  \to \ias(m+n-1)\]
is given by
\[\mu\ucr{i}  \nu=\mu\circ (1_1, \cdots,
\underset{i}{\nu}, \cdots, 1_1)\]
for $\mu\in \ias(m), \nu\in \ias(n)$, $m\ge 1, n\ge 0$
and $1\le i\le m$.

The operad $\ias$ encodes unital associative algebras,
namely, a unital associative algebra is exactly a
$\ias$-algebra. Let $(A, \mu, u)$ be a unital associative
algebra. One can define an operad morphism
$\gamma=(\gamma_n)\colon \ias \to \mathcal{E}nd_A$
given by $\gamma_0(1_0)=u$ and $\gamma_2(1_2)=\mu$,
where $\mathcal{E}nd_A$ is the endomorphism operad
of the vector space $A$, see \cite[Section 5.2.11]{LV}.
Each $\theta=\sum_{\sigma\in \S_n} c_\sigma \sigma\in
\Bbbk\S_n$ gives an $n$-ary operation on $A$, 
\begin{align}
\notag
\gamma_n(\theta)\colon \quad
A^{\otimes n} \to A, \quad
\gamma_n(\theta)(a_1\otimes\cdots\otimes a_n)
=\sum_{\sigma\in \S_n} c_\sigma a_{\sigma^{-1}(1)}
\cdots a_{\sigma^{-1}(n)}, \quad {\text{for
$a_1,\cdots, a_n\in A$}}.
\end{align}

Next we work out some connections between T-ideals and 
operadic ideals of $\ias$.

Denote by $\Bbbk\lr{x_1, \cdots, x_n}$ the free associative
algebra in noncommutative indeterminants $\{x_1, \cdots, x_n\}$
over $\Bbbk$. Observe that a PI algebra $A$ always satisfies a 
multilinear polynomial identity of degree $\le d$ if $A$ satisfies 
an identity of degree $d$, see \cite[Proposition 13.1.9]{MR}
or \cite[Theorem 1.3.7]{GZ}. A \textit{multilinear polynomial}
of degree $n$ is a nonzero element $f(x_1, \cdots, x_n)\in
\Bbbk\langle x_1, \cdots, x_n \rangle$ of the form
\[f(x_1, \cdots, x_n)=\sum_{\sigma\in \S_n}
c_\sigma x_{\sigma^{-1}(1)} \cdots x_{\sigma^{-1}(n)}\]
for some $c_{\sigma}\in \Bbbk$. The method of
multilinearization actually plays a very important role in
the study of the identities of a PI-algebra.

Based on the following observation, one can study
PI-algebras in the language of operads. The following
lemma is a folklore.

\begin{lem}
\label{xxlem2.1}
Let $A$ be an associative algebra. Then $A$ is a PI
algebra if and only if $A$ is a \; $\ias/\I$-algebra
for some nonzero operadic ideal $\I$ of $\ias$.
\end{lem}

\begin{proof}
Let $\I\neq 0$ be an ideal of\; $\ias$ and $A$ a\; 
$\ias/\I$-algebra with the operadic morphism
$\bar \gamma\colon \ias/\I\to \mathcal{E}nd_A$.
Clearly, $A$ is an associative algebra. Then for each
nonzero element
$\theta=\sum_{\sigma\in \S_n}c_\sigma \sigma\in \I(n)$,
the algebra $A$ satisfies the following multilinear polynomial
\[f_\theta(x_1, \cdots, x_n)=\sum_{\sigma\in \S_n}
c_\sigma x_{\sigma^{-1}(1)}\cdots x_{\sigma^{-1}(n)},\]
since
\begin{align*}
f_\theta(a_1, \cdots, a_n)
=&\bar \gamma(\theta+\I)(a_1, \cdots, a_n)=0
\end{align*}
for any $a_1, \cdots, a_n\in A$. Conversely, let $A$ be a
PI-algebra satisfying a multilinear polynomial of the form
$$f(x_1, \cdots, x_n)=\sum_{\sigma\in\S_n}
c_\sigma x_{\sigma^{-1}(1)}\cdots x_{\sigma^{-1}(n)},$$
where $c_{\sigma}\in \Bbbk$. Clearly, $A$ is a $\ias$-algebra.
Suppose that $\gamma\colon \ias\to \mathcal{E}nd_A$ is the
corresponding operadic morphism. Then 
$\theta_f=\sum_{\sigma\in \S_n} c_\sigma \sigma\in \Ker \gamma$ 
and therefore $\Ker\gamma$ is a nonzero operadic ideal of $\ias$. 
It follows that $A$ is an algebra over the quotient operad 
$\ias/\Ker\gamma$.
\end{proof}

Denote $\I_A:=\Ker\gamma$ as in the proof of Lemma \ref{xxlem2.1}.
In this case, we say that $A$ is a PI-algebra associated to the
operadic ideal $\I_A$. Clearly, the operadic ideal $\I_A$ is 
the maximal operadic ideal $\I$ of $\ias$ such that $A$ is an 
algebra over $\ias/\I$.

Let $\I$ be a nonzero operadic ideal of $\ias$ and
$\mathcal{A}_\I=\ias/\I$. Suppose that $V$ is a vector space 
over $\Bbbk$. Recall the free $\mathcal{A}_\I$-algebra 
$\mathcal{A}_\I(V)$ with
\[\mathcal{A}_\I(V)= \bigoplus_{k\ge0}\mathcal{A}_\I(V)_k,
\quad {\rm with} \quad \mathcal{A}_\I(V)_k
= \mathcal{A}_\I(k)\ot_{\k\S_k} V^{\ot k},\]
where the left action of $\S_k$ on $V^{\otimes k}$ is given by
\[\sigma\cdot (v_1\otimes \cdots \otimes v_n)\colon
=v_{\sigma^{-1}(1)} \otimes \cdots \otimes v_{\sigma^{-1}(n)},\]
and the multiplication
\[\mathcal{A}_\I(V)_m \otimes \mathcal{A}_\I(V)_n
\to \mathcal{A}_\I(V)_{m+n}\]
given by
\[[\bar \mu, u_1, \cdots, u_m]\cdot [\bar\nu, v_1, \cdots, v_n]
\colon=[\overline{1_2\circ (\mu, \nu)},
u_1, \cdots, u_m, v_1, \cdots, v_n].\]

Clearly, each nonzero element in $\I(n)$ gives an identity
satisfied by $\mathcal{A}_\I(V)$. Moreover, the free
$\mathcal{A}_\I$-algebra $\mathcal{A}_\I(V)$ is a PI-algebra
associated to $\I$. Take different  vector space $V$, one
can obtain different PI-algebra $\mathcal{A}_\I(V)$ associated
to $\I$.

Let $\Bbbk\langle X \rangle$ be the free algebra generated
by the set $X=\{x_i\}_{i\in \N_+}$. Recall that an ideal $H$
of $\Bbbk\langle X \rangle$ is called a \textit{T-ideal} if
$\varphi(H)\subset H$ for every endomorphism $\varphi$ of
$\Bbbk\langle X \rangle$. Let $A$ be a PI algebra. The set
$\Id(A)$ of all polynomial identities of $A$ in
$\Bbbk\langle X \rangle$ is a T-ideal.
Conversely, if $H$ is a T-ideal of $\Bbbk\langle X \rangle$,
then $\Bbbk\langle X \rangle/H$ is a ``free'' or ``universal''
PI algebra in some sense. To be precise, if $A$ is a PI
algebra such that $\Id(A)\supseteq H$, then for any set
mapping $\phi\colon X \to A$, there exists a unique algebra
homomorphism $\bar \phi\colon \Bbbk\lr{X}/H \to A$ such that
the following diagram
\begin{align*}
\xymatrix{X \ar[r]^-{\phi}\ar[dr] &  A\\
	& \Bbbk\langle X \rangle/H\ar@{-->}[u]_{\bar\phi} } \tag{E2.1.1}
\end{align*}
commutes, see \cite[Theorem 2.2.17]{AGPR}. It is easily seen that 
the ``free'' algebra $\Bbbk\lr{X}/H$ is just the free 
$\ias/\I$-algebra $(\ias/\I)(V)$, where $V$ is the vector space 
spanned by $X=\{x_i\}_{i\in \N_+}$, and $\I$ is the kernel
of the structure morphism $\gamma\colon \ias \to 
\mathcal{E}nd_{\Bbbk\lr{X}/H}$ of $\Bbbk\lr{X}/H$ as a 
$\ias$-algebra, since the free algebra $\ias/\I(V)$ also 
guarantees the existence of the above commutative diagram.
Therefore, there is a correspondence between the classes of
PI algebras and the T-ideals of $\Bbbk\langle X\rangle$.

A basic question about finite generation of T-ideals was posed 
by Specht \cite{Sp} in 1950. In order to avoid confusion with 
finitely generated as an ideal, a finitely generated T-ideal in 
the class of T-ideals is usually called finitely based. In 1987 
Kemer gave an affirmative answer \cite{Ke1, Ke2}. Further
discussions can be found in \cite{AGPR, AKK, KR, Pr2}.

\begin{thm}\cite[Theorem 2.4]{Ke2}
\label{xxthm2.2}
Every associative algebra {\rm{(}}with or without unit{\rm{)}}
has a finite basis of identities.
\end{thm}

Kemer's proof is based on some structure theory of superidentities
of superalgebras and certain graded tensor products with the
Grassmann algebra.

For each $n\in \N$, we denote $V_n$ the subspace of
$\Bbbk\lr{x_1, \cdots, x_n}$ spanned by
$x_{\sigma^{-1}(1)}\cdots x_{\sigma^{-1}(n)}, \sigma\in \S_n$,
which consists of all multilinear polynomials of degree $n$ in
the indeterminants $x_1, \cdots, x_n \in X$. Observe that
$V_n$ admits the right $\Bbbk\S_n$-action given by
\[(x_{i_1}\cdots x_{i_n})\ast \tau\colon=x_{\tau^{-1}(i_1)}
\cdots x_{\tau^{-1}(i_n)}\]
for $\tau\in \S_n$, and
\[\Phi_n\colon \As(n) \to  V_n, \quad \quad \sigma  \mapsto
 x_{\sigma^{-1}(1)}\cdots x_{\sigma^{-1}(n)} \]
is an isomorphism of right $\Bbbk\S_n$-modules.

\begin{lem}
\label{xxlem2.3}
Let $A$ be a PI algebra and $V_n(A)$ be the subspace of $V_n$
consisting of all multilinear identities of $A$ in
$\Bbbk\lr{x_1, \cdots, x_n}$. Denote $\I(n)=\Phi_n^{-1}(V_n(A))$
for each $n\in \N_+$. Then $\I=(\I(n))_{n\in \N_+}$ is an
operadic ideal of $\ias$.
\end{lem}

\begin{proof}
It is easily seen that $V_n(A)$ is invariant under the right
$\S_n$-action, and so is $\I(n)$. Let $\mu\in \ias(m), \nu\in
\ias(n)$, $1\le i\le m$. If $\mu\in \I(m)$ or $\nu\in \I(n)$,
then for $a_1, \cdots, a_{m+n-1}\in A$, we have
$$
(\Phi_{m+n-1}(\mu\ucr{i} \nu))(a_1, \cdots, a_{m+n-1})
=\Phi_m(\mu)(a_1, \cdots, a_{i-1}, \Phi_n(\nu)
 (a_i, \cdots, a_{i+n-1}), a_{i+n}, \cdots, a_{m+n-1})
=0.
$$
Therefore, $\Phi_{m+n-1}(\mu\ucr{i} \nu)\in V_{m+n-1}(A)$ and
$\mu\ucr{i}\nu\in \I(m+n-1)$. It follows that $\I$ is an
operadic ideal of $\ias$.
\end{proof}

\begin{lem}
\label{xxlem2.4}
Let $\I=(\I(n))_{n\in \N}$ be an operadic ideal of $\As$
and $\Bbbk\langle X\rangle$ the free algebra generated by
$X=\{x_n\}_{n\in \N^+}$. Put
\begin{equation}
\label{E2.4.1}\tag{E2.4.1}
J_n\colon=\{\Phi_n(\theta)(f_1, \cdots, f_n)\mid \theta
\in \I(n), f_1, \cdots, f_n\in \Bbbk\langle X \rangle\}
\end{equation}
for each $n\in \N$. Suppose that $J$ is the ideal of
$\Bbbk\langle X \rangle$ generated by $\cup_{n\in \N}J_n$.
Then $J$ is a T-ideal of $\Bbbk\langle X\rangle$.
\end{lem}

\begin{proof}
Let $\varphi$ be an endomorphism of the free algebra
$\Bbbk\langle X \rangle$. For any $\Phi_n(\theta)(f_1,
\cdots, f_n)\in J_n$, we have
\[\varphi(\Phi_n(\theta)(f_1, \cdots, f_n))=\Phi_n(\theta)
(\varphi(f_1), \cdots, \varphi(f_n))\in J_n.\]
Therefore, $\varphi$ sends a generator onto a generator of
$J$, and $J$ is a $T$-ideal of $\Bbbk\langle X\rangle$.
\end{proof}

We fix the following notations.

\begin{notation}
\label{xxnot2.5}
Let $A$ be a PI algebra and $\I$ be a proper operadic ideal
of $\As$.
\begin{enumerate}
\item[(1)]
Let $\Psi(A)$ denote the operadic ideal of $\As$ constructed in
Lemma \ref{xxlem2.3}.
\item[(2)]
Let $J$ be a T-ideal of $\Bbbk\lr{X}$. By abuse of notation,
$\Psi(\Bbbk\lr{X}/J)$ is also denoted by $\Psi(J)$. 
Given a T-idea $J$, the operadic ideal $\Psi(J)$ of $\ias$ is 
defined by $\Psi(J)(n)=\Phi^{-1}_{n}(V_n(\Bbbk\lr{X}/J))$ for 
all $n$.  
\item[(3)]
Let $\Omega(\I)$ denote the T-ideal of $\Bbbk\lr{X}$ constructed
in Lemma \ref{xxlem2.4}. Given an operadic ideal $\I$ of 
$\ias$, then $\Omega(\I)$ is generated by $\cup_n \Phi_n(\I(n))$ 
as a T-ideal. 
\end{enumerate}
\end{notation}

Two sets of polynomials are said to be {\it equivalent}
if they generate the same T-ideal. As usual we assume that
$\Bbbk$ is a base field of characteristic zero.

\begin{lem}
\cite[Theorems 1.3.7 and 1.3.8]{GZ}
\label{xxlem2.6}
Every nonzero polynomial
$f\in \Bbbk\lr{X}$ is equivalent to a finite set
$\{f_1,\cdots, f_w\}$ of multilinear polynomials with
$\deg f_i\leq \deg f$.
\end{lem}

By Theorem \ref{xxthm2.2} and Lemma \ref{xxlem2.6}, the
following result is obvious.

\begin{cor}
\label{xxcor2.7}
Every proper T-ideal is generated by finitely many
multilinear polynomials as a T-ideal.
\end{cor}

Let $f$ be a multilinear polynomial of degree $n$ in
$\Bbbk\lr{x_1, \cdots, x_n}$.
We use $I_f$ to denote the ideal of $\Bbbk\lr{X}$ of the form
$\{\sum_{i=1}^m g_if(u_{i1}, \cdots, u_{in})h_i\mid g_i, 
h_i, u_{ij} \in \Bbbk\lr{X}, i=1, \cdots, m, j=1, \cdots, n\}$.

\begin{lem}
\label{xxlem2.8}
Let $W$ be a set of multilinear polynomials $\{f\}$ where
each $f$ is a multilinear polynomial of degree $n$ in
$\Bbbk\lr{x_1, \cdots, x_n}$ for some $n$. Let $\lr{W}_T$ be
the T-ideal of $\Bbbk\lr{X}$ generated by $W$ as a
T-ideal. Then $\lr{W}_T=\sum_{f\in W} I_f$.
\end{lem}

\begin{proof} It is easy to reduce to the case when $W$ is the singleton
$\{f\}$. In this case we need to show that $\lr{f}_T:=\lr{W}_T=I_f$.
Clearly, $gf(u_1, \cdots, u_n)h\in \lr{f}_T$ for all $g, h, u_i\in
\Bbbk\lr{X}$, $i=1, \cdots, n$, and therefore $I_f\subset \lr{f}_T$.
It suffices to show that $I_f$ is a T-ideal. It is easily seen that
$I_f$ is an ideal of $\Bbbk\lr{X}$. Suppose that $\varphi$ is an
arbitrary endomorphism of $\Bbbk\lr{X}$. Then for all
$g_i, h_i, u_{ij}\in \Bbbk\lr{X}, i=1, \cdots, m, j=1, \cdots, n$,
we have
\[\varphi(\sum_{i=1}^m g_if(u_{i1}, \cdots, u_{in})h_i)
=\sum_{i=1}^m\varphi(g_i)f(\varphi(u_{i1}), \cdots,
\varphi(u_{in}))\varphi(h_i)\in I_f.\]
It follows that $I_f$ is a T-ideal and $\lr{f}_T=I_f$.
\end{proof}

Here is an intermediate step.

\begin{thm}
\label{xxthm2.9}
The pair $(\Psi,\Omega)$ defined in Notation \ref{xxnot2.5}
induce an inclusion-preserving one-to-one correspondence between
the set of proper T-ideals of $\Bbbk\lr{X}$ and the set of proper
operadic ideals of $\As$.
\end{thm}

\begin{proof}
By Lemmas \ref{xxlem2.3} and \ref{xxlem2.4},
we have inclusion-preserving maps
$$\Psi: \quad 
\{{\text{proper T-ideals of $\Bbbk\lr{X}$}}\}
\to \{{\text{proper operadic ideals of $\As$}}\}$$
and
$$\Omega: \quad 
\{{\text{proper operadic ideals of $\As$}}\}
\to
\{{\text{proper T-ideals of $\Bbbk\lr{X}$}}\}.$$

First we prove that $\Omega(\Psi(J))=J$ if $J$ is a proper
T-ideal of $\Bbbk\lr{X}$. By Corollary \ref{xxcor2.7},
$J$ is generated by finitely many multilinear polynomials,
say $\{f_1,\cdots, f_s\}$, as a T-ideal. Let $A:=\Bbbk\lr{X}/J$.
Since $J$ is a T-ideal, every element in $J$ is an identity of
$A$. In particular, $f_i\in V_n(A)$ if $f_i$ has degree $n$. By
Lemma \ref{xxlem2.3}, $\Phi_n^{-1}(f_i)\in \I(n)$ where
$\I=\Psi(J)$. By Lemma \ref{xxlem2.4}, $f_i\in \Omega(\I)$.
Consequently, $J\subseteq \Omega(\I)=\Omega(\Psi(J))$. Conversely,
let $J'=\Omega(\Psi(J))$, we need to show that $J'\subseteq J$.
Since both $J$ and $J'$ are T-ideals, by Lemma \ref{xxlem2.4},
it suffices to show that $J_n\subseteq J$ where $J_n$ is defined as
in \eqref{E2.4.1}. Again, by the fact that $J$ is a T-ideal, it remains
to show that $\Phi_n(\theta) \in J$ for all $\theta\in \I(:=\Psi(J))$.
By the definition of $\I$ in Lemma \ref{xxlem2.3},
$\Phi_n(\theta) \in V_n(A)$ where $A$ is $\Bbbk\lr{X}/J$. This
implies that $\Phi_n(\theta)$ is an identity of $A$. Therefore
$\Phi_n(\theta) \in J$ as required.

Next we show that $\Psi(\Omega(\I))=\I$ for any proper 
operadic ideal $\I$ of $\As$. Let $J:=\Omega(\I)$. By the 
proof of Lemma \ref{xxlem2.4}, $J$ is the 2-sided 
ideal of $\Bbbk\lr{X}$ generated by $\cup_{n\in \N} J_n$. 
Then, for every $\theta \in \I(n)$, 
$\Phi_n(\theta)=\Phi_n(\theta)(x_1,\cdots,x_n)$ is in 
$J_n\subseteq J$. Thus $\Phi_n(\theta)$ is in $V_n(\Bbbk\lr{X}/J)$.
By Lemma \ref{xxlem2.3}, $\theta\in \Psi(J)$.
This proves that $\I\subseteq \Psi(\Omega(\I))$. Conversely, let
$\I':=\Psi(\Omega(\I))$, we need to show that
$\I'\subseteq \I$. Let $\theta$ be in $\I'$.  By definition,
$\theta=\Phi^{-1}_n(f_{\theta})$ where $f_{\theta}$ is some 
multilinear identity of $\Bbbk\lr{X}/J$, or equivalently, 
$f_{\theta}$ is multilinear and $f_{\theta}\in J$ as $J$ is 
a T-ideal. By the proof of Lemma \ref{xxlem2.4}, there 
exist $g_{ki}, h_{ki}, u_{ki,1},\cdots, u_{ki,n_k} \in \Bbbk\lr{X}$ 
such that
\[f_\theta=\sum_{k,i} g_{ki}f_k(u_{ki,1}, \cdots, 
u_{ki, n_{k}})h_{ki}\]
where $f_k=\Phi_{n_k}(\theta_k)$ for some $\theta_k\in \I(n_k)$
(and the above sum is a finite sum). Let $l$ denote $n_k$. 
Without loss of generality, we may assume that $g_{ki}, h_{ki}, 
u_{ki,1},\cdots u_{ki,l}$ are monomials. Since $f_\theta$ is a 
multilinear polynomial in $\Bbbk\lr{x_1, \cdots, x_n}$, we may 
further assume that each term $g_{ki}f_k(u_{ki,1}, \cdots, 
u_{ki,l})h_{ki}$ is a multilinear polynomial in the 
indeterminants $x_1, \cdots, x_n$, and for each pair $(k,i)$, 
$g_{ki}, h_{ki}, u_{ki,1},\cdots, u_{ki,l}$ are in different 
sets of indeterminants. We need only show that
$\Phi^{-1}_n(g f_k(u_{1}, \cdots, u_{l})h)\in \I(n)$, where
$(g, h, u_1, \cdots, u_{l})$ are sets of monomials of the 
form $(g_{ki}, h_{ki}, u_{ki,1},\cdots, u_{ki,l})$, can be 
generated by $\theta_k\in \I(l)$ for all $k$. We denote
$$\begin{aligned}
g&=ax_{\sigma^{-1}(1)}\cdots x_{\sigma^{-1}(r)}, \\
u_{1}&=b_{1}x_{\sigma^{-1}(r+1)}\cdots x_{\sigma^{-1}(r+s_1)},\\
& \cdots, \\
u_{l}&=b_{l}x_{\sigma^{-1}(r+s_1+\cdots+s_{l-1}+1)}\cdots
x_{\sigma^{-1}(r+s_1\cdots+s_l)},\\
h&=cx_{\sigma^{-1}(r+s_1\cdots+s_l+1)}\cdots x_{\sigma^{-1}
(r+s_1\cdots+s_l+t)}
\end{aligned}
$$ for some $\sigma\in \S_n$ and some nonzero scalars 
$a, b_1,\cdots, b_l, c$ in $\Bbbk$. Clearly, $r+s_1+\cdots+s_l+t=n$.
Then
\begin{align*}
\Phi_n^{-1}(g f_k(u_{1}, \cdots, u_{l})h)=ab_1\cdots b_l
c(1_3\circ (1_r, \theta_k\circ (1_{s_1}, \cdots, 1_{s_l}), 1_t))
\ast \sigma\in \lr{\theta_k}(n),
\end{align*}
where $\lr{\theta_k}$ is the operadic ideal of $\ias$
generated by $\theta_k$. It follows that
$\theta\in \sum_k \lr{\theta_k}(n)$
where each $\theta_k\in \I(n_k)$. Therefore $\theta\in \I(n)$
as required.
\end{proof}

\begin{rmk}
\label{xxrem2.10}
If a T-ideal $J$ is generated by multilinear identities 
$\{f_i\}_{i=1}^{s}$ as a T-ideal, then $\Psi(J)$ is generated 
by $\{\Phi^{-1}_{\deg f_i}(f_i)\}_{i=1}^s$ as an operadic
ideal of $\ias$. To see this, let $\I$ be $\Psi(J)$ and $\I'$ be the
operadic ideal of $\ias$ generated by 
$\{\Phi^{-1}_{\deg f_i}(f_i)\}_{i=1}^s$. It follows from 
Lemma \ref{xxlem2.3} that $\I'\subseteq \I$. By 
Theorem \ref{xxthm2.9}, $\Omega(\I')\subseteq \Omega(\I)=J$.
By Lemma \ref{xxlem2.4}, each $f_i$ is in 
$\Omega(\I')$. Thus $\Omega(\I')\supseteq J$. This forces
that $\I'=\I$ and consequently, $\I$ is generated by 
$\{\Phi^{-1}_{\deg f_i}(f_i)\}_{i=1}^s$ as an operadic
ideal of $\ias$.
\end{rmk}

Note that a T-ideal of the free nonunital associative 
algebra $\Bbbk\lr{X}_{+}$ can be defined similarly.
Similarly one can show the following.

\begin{thm}
\label{xxthm2.11}
There is an inclusion-preserving one-to-one correspondence 
between the set of proper T-ideals of $\Bbbk\lr{X}_{+}$ and 
the set of proper operadic ideals of $\nias$.
\end{thm}

Kemer proved that every proper T-ideal of $\Bbbk\lr{X}$ is
finitely generated as a T-ideal in \cite{Ke1, Ke2}. Applying the
one-to-one correspondence between T-ideals of $\Bbbk\lr{X}$ and
operadic ideals of $\ias$ [Theorem \ref{xxthm2.9}], we have the
following consequences.

\begin{thm}
\label{xxthm2.12}
\begin{enumerate}
\item[(1)]
The operad\; $\ias$ is noetherian, that is, the set of
operadic ideals of\; $\ias$ satisfies the ascending chain
condition.
\item[(2)]
Every ideal of\; $\ias$ is finitely generated as an operadic
ideal.
\end{enumerate}
\end{thm}

\begin{proof}
(1) Since every T-ideal of $\Bbbk\lr{X}$ is finitely generated
as a T-ideal [Theorem \ref{xxthm2.2}], the set of T-ideals of
$\Bbbk\lr{X}$ satisfies the ascending chain condition.
The assertion follows from Theorem \ref{xxthm2.9}.

(2) This follows from part (1) by a standard noetherian argument.
\end{proof}

The following lemma is needed.

\begin{lem}\cite[Lemma 2.15(2)]{BYZ}
\label{xxlem2.13}
Let $\I$ be a finitely generated operadic ideal of $\ias$. Then 
there exists $\theta\in \I(n)$ for some $n\ge 0$ such that 
$\I=\lr{\theta}$.
\end{lem}

Now we are ready to prove Theorem \ref{xxthm1.2}.

\begin{proof}[Proof of Theorem \ref{xxthm1.2}]
(1) This is Theorem \ref{xxthm2.12}(1).

(2) This follows from Lemma \ref{xxlem2.13} and Theorem 
\ref{xxthm2.12}(2).
\end{proof}

In general the description of a T-ideal is very difficult even
if every proper T-ideal is finitely generated \cite{Ke1, Ke2}. 
In fact it is quite difficult to deduce the generators from
a given T-ideal. An effective way is to study the multilinear
polynomials in a T-ideal since every identity is equivalent to
a system of multilinear polynomials. Here is a small improvement 
of the original Kemer's theorem [Theorem \ref{xxthm2.2}] when 
we consider associative algebras with unit. 

\begin{cor}
\label{xxcor2.14}
Every proper T-ideal corresponding to an associative algebra
with unit is generated by one multilinear polynomial as
a T-ideal.
\end{cor}

\begin{proof}
Let $J$ be a proper T-ideal of $\Bbbk\lr{X}$.
By Theorem \ref{xxthm2.9}, $J=\Omega(\I)$ where $\I$
is a proper operadic ideal of $\ias$. By part (2),
$\I$ is generated by an element $\theta\in \I(n)$.
Let $J'$ be the T-ideal of $\Bbbk\lr{X}$ generated
by $\Phi_n(\theta)$. It remains to show that $J=J'$.
By the proof of Lemma \ref{xxlem2.4}, $J'\subseteq J$. Since
$\Phi_n(\theta)$ is an identity of $\Bbbk\lr{X}/J'$,
it follows from the proof of Lemma \ref{xxlem2.3}
that
$$\theta\in \Psi(J') \subseteq \Psi(J)=\Psi(\Omega(\I))
=\I.$$
Since $\I$ is generated by $\theta$, we obtain that
$\I=\Psi(J')=\Psi(J)$. Now Theorem \ref{xxthm2.9}
implies that $J'=J$ as required.
\end{proof}

It is well-known that Corollary \ref{xxcor2.14} fails for algebras
without unit. Note that we do not provide a new proof of Kemer's 
theorem. Corollary \ref{xxcor2.14} is in the same spirit as a 
result of Razmyslov \cite{Ra} and Procesi \cite{Pr1} which states 
that all trace identities for the full matrix algebras are 
generated by a single trace identity, see also \cite{IKM}.

\begin{rmk}
\label{xxrem2.15}
Using Theorems \ref{xxthm2.11} and \ref{xxthm2.2},
one sees that the operad $\nias$ is also noetherian on 
operadic ideals.
\end{rmk}

\end{document}